\documentclass [10pt]{amsart}
\usepackage{amsmath,amsfonts,amsthm,amssymb,pxfonts}
\usepackage{Tabbing}
\usepackage{lastpage}
\usepackage{mathtools}
\usepackage{esint}
\usepackage{color}
\usepackage{mathrsfs}
\usepackage{hyperref}

\newtheorem{theorem}{Theorem}
\newtheorem{definition}[theorem]{Definition}

\newtheorem{lemma}[theorem]{Lemma}

\newtheorem{corollary}[theorem]{Corollary}
\newtheorem{proposition}[theorem]{Proposition}
\newtheorem{remark}[theorem]{Remark}

\DeclareMathOperator{\ad}{ad}

\begin{document}

\title[Stability of the cubic nonlinear Schr\"{o}dinger equation on Irrational Torus]{Stability of the cubic nonlinear Schr\"{o}dinger equation on Irrational Torus}

\author{Gigliola Staffilani}
 \address{Department of Mathematics,  Massachusetts Institute of Technology,  77 Massachusetts Ave,  Cambridge,  MA 02139-4307 USA. email:
gigliola@math.mit.edu.}
\thanks{
 Gigliola Staffilani was partially supported by NSF   grants  DMS 1362509 and DMS 1462401, the Simons Foundation and the John Simon Guggenheim Foundation.}
\author{Bobby Wilson }
\address{Department of Mathematics,  Massachusetts Institute of Technology,  77 Massachusetts Ave,  Cambridge,  MA 02139-4307 USA. email:
blwilson@mit.edu.}



\begin{abstract}
A characteristic of the defocusing cubic nonlinear Schr\"odinger equation (NLSE), when defined so that the space variable is the multi-dimensional square (hence rational) torus, is that  there exist solutions that start with arbitrarily small norms Sobolev norms and  evolve to develop  arbitrarily large modes at later times; this  phenomenon is recognized as a weak  energy transfer to high modes for the NLSE, \cite{CKSTT10} and \cite{CF12}. In this paper, we show that when the system is considered on  an irrational torus, energy transfer  is more difficult to detect.
\end{abstract}

\maketitle 

\tableofcontents

%
%
%

\section{Introduction}\label{sec.intro}

Consider the Cauchy problem for the defocusing cubic Nonlinear Schr\"odinger Equation (NLSE):

	\begin{align}\label{NLS}
		\left\{ \begin{array}{ll} i\partial_t \psi= \Delta \psi-|\psi|^2\psi, & x\in \mathbb{T}^2, t \in \mathbb{R}\\
						\psi(0)= \psi_0 \in H^s(\mathbb{T}^2),   \end{array} \right.
	\end{align}
where $\mathbb{T}^2$ is the two dimensional torus and  $H^s(\mathbb{T}^2)$ is the usual Sobolev space of index $s$.
The significance of considering the case in which our spatial variable, $x$, belongs to the two-dimensional torus stems from an important dichotomy between the case when space is one-dimensional and when it is higher than one. This dichotomy exists in both the scope of what we know about each case and the dynamical behavior of each system.  It is well known that the one-dimensional defocusing cubic NLSE, both periodic and non,  is an integrable system, see for example \cite{FT}, and as a consequence the rigidity of this structure contributes to the uniform in  time control of any Sobolev norm of the solutions. In other words 
one can combine the conservation laws at  order $k\in \mathbb{N}$ to show that if for the initial data $\|\psi_0\|_{H^k}<\infty$, then for all $t\in \mathbb{R}$ we have that the solution satisfies $\|\psi(t)\|_{H^k}\leq C.$
In particular this control shows that there is no energy transfer from low to high modes. 

In two-dimensions the integrability is lost, and hence the question of    energy transfer to high modes is more subtle.  In \cite{CKSTT10}, the authors demonstrated that for any $s >1$  and for any two positive real numbers $\varepsilon$, $M$, there exists a solution $w$ to \eqref{NLS},  such that the Sobolev norms $\|w(0)\|_s<\varepsilon$ and $\|w(T)\|_s>M$, for $T$ large enough. Later Hani \cite{H11} generalized this energy transfer phenomenon   in his thesis.  

A question that remains is whether or not one can quantify the rate of the kind of energy transfer   recalled above.  For example, in \cite{B98}, Bourgain shows that the linear Schr\"odinger equation with a time-dependent potential can produce solutions whose Sobolev norms grow on the order of a power $\log |t|$, see also \cite{W08} and \cite{MR17}.   Furthermore, Carles and Faou \cite{CF12} showed quantitative energy transfer  results using a Birkhoff normal form transformation coupled with standard dynamical systems techniques.   In all these works the torus $\mathbb{T}^2$ is assumed to be a square torus, or more generally a rational torus, according to the following definition.
For $\omega= (\omega_1, \omega_2,..., \omega_d) \in \mathbb{R}^d_+$, we say that $\omega$ is an irrational vector if 
	\begin{align*}
	\omega \cdot m \neq 0 \mbox{ for any } m \in \mathbb{Z}^d,
	\end{align*}
	otherwise we say that it is rational.
  If $\mathbb{T}:= \mathbb{R}/\mathbb{Z}$ is the standard torus, then let 
	\begin{align*}
		\mathbb{T}^d_{\omega}= \prod_{i=1}^d \mathbb{R}/(\omega_i \mathbb{Z}).
	\end{align*}
 We say that $\mathbb{T}^d_{\omega}$ is rational or irrational according to the rationality or irrationality  of $\omega$. 
 
In this work we show that the mechanisms to prove  energy transfer in the two works \cite{CF12} and \cite{CKSTT10} recalled above are not enough to prove energy transfer  when the torus is irrational. We are not claiming that in this case long time energy transfer  should not  be expected, but that a more sophisticated constructive method may be needed. 
 
 Before we move forward we should mention that until recently the assumption of rationality of the torus, where a problem such as  \eqref{NLS}  was posed, had been essential in order to prove well-posedness for example. In fact until the recent work of Bourgain and Demeter \cite{BD}, where the full range of sharp\footnote{Up to an $\epsilon$ loss of derivative, recovered in some cases in the work of Killip and Visan \cite{KV16}.} Strichartz estimates where proved for any torus $\mathbb{T}^d_{\omega}$ as a consequence of the $l^2$ decoupling conjecture, some sharp\footnote{For some Strichartz estimates on  irrational tori with loss of derivative see also \cite{B007, CW10, GOW}.} Strichartz estimates had been established by Bourgain \cite{B93} only for rational tori using techniques from analytic number theory. Following the work \cite{BD}, Fan, Staffilani, Wang, and Wilson \cite{FSWW17},   Deng, Germain, and Guth \cite{DGG17},  Deng, Germain \cite{DG17} and  
Deng \cite{D17}, also considered \eqref{NLS} when $\mathbb{T}^d_\omega$ is  irrational. More precisely
for fixed, irrational $\omega$, consider the periodic NLSE initial value problem 
\begin{align}\label{irrNLS}
\left\{ \begin{array}{ll} i\partial_t \psi= \Delta \psi-|\psi|^2\psi, & x\in \mathbb{T}_{\omega}^2, t \in \mathbb{R}\\
						\psi(0)= \psi_0 \in H^s(\mathbb{T}_{\omega}^2),   \end{array} \right.
\end{align}
where  $x\in\mathbb{T}_{\omega}^d$ and $\Delta$ is the Laplace-Beltrami operator for the irrational torus\footnote{One must make the decision to scale the torus or the Laplace operator and we will make sure not to scale more than once.}. In \cite{FSWW17}, the authors proved some improved Strichartz estimates, and as a consequence they where able to extend to \eqref{irrNLS} the same results of local and global   well-posedness in the Sobolev space $H^s(\mathbb{T}^2_{\omega})$ as the ones already available for \eqref{NLS} in rational tori, see \cite{B93, DPST07}. Deng, Germain, and Guth \cite{DGG17} proved that the dynamics of the nonlinear Schr\"odinger  flow are, in some sense, ``better" behaved by proving   that Strichartz estimates on  irrational tori hold  for   longer time than on the standard torus, and quantified  the difference between the two cases with the irrationality of $\omega$. Deng in \cite{D17} and Deng-Germain in \cite{DG17}, proved that for certain nonlinear Schr\"odinger  equations on irrational tori the degree of the polynomial bounds for higher Sobolev norms are better than the ones for rational tori, see also \cite{B96, St97, B004, S1-11, S2-11, S12, CKO12}. In a sense these results  support the intuition that when the flow is allowed to evolve on an irrational torus the effects from the boundary  are less severe and  the flow has more room to evolve. This phenomenon should be interpreted  as a very weak nonlinear dispersion of the flow. In the spirit of \cite{DGG17, DG17, D17}, we give here one more evidence that on irrational tori the nonlinear Schr\"odinger flow is in a sense more regular. We will  show in fact that the type of energy transfer results  demonstrated in \cite{CF12} are  not present in the irrational case to the extent that it is in the rational one.  

For $\psi \in H^s(\mathbb{T}_{\omega}^d)$, $s\geq 1$, \eqref{irrNLS} has the Hamiltonian form
	\begin{align*}
	 H(\psi) = \tfrac{1}{2}\int |\nabla \psi|^2+ \tfrac{1}{4} \int |\psi|^{4}.
	\end{align*}
In this case, we can also define the Hamiltonian with respect to the Fourier coefficients.  Let $e(y):= e^{2\pi i y}$, then 
	\begin{align*}
		\hat{\psi}_k:=\hat{\psi}(k):= \int_{\mathbb{T}_{\omega}^d} \psi(x) e( \langle x, k\rangle_{\omega}) \,dx=\int_{\mathbb{T}_{\omega}^d} \psi(x) e\left( \sum_{i=1}^d k_i\omega_ix_i \right) \,dx
	\end{align*}
for $k \in \mathbb{Z}^d$  as
	\begin{align}\label{ham}
		H( \hat{\psi}) &= \tfrac{1}{2} \sum_{k \in \mathbb{Z}^d} \lambda_k |\hat{\psi}_k|^2 + \tfrac{1}{4}\sum_{ k_1+k_2= h_1+h_2} \hat{\psi}_{k_1} \hat{\psi}_{k_2}\bar{\hat{\psi}}_{h_1}\bar{\hat{\psi}}_{h_2}\\
				&=:H_0 + P,
	\end{align}
where $\lambda_k := \sum_{i=1}^d \omega_i^2 k_i^2$.
 
 \medskip
To state the main theorem below let us denote with $\|\cdot\|_s$ the norm in $H^s(\mathbb{T}_{\omega}^d)$, and  let us denote with $Q_M$ the $2M$-length box in $\mathbb{Z}^2$ centered at the origin , $Q_M:=\left\{k \in \mathbb{Z}^2: \|k\|_{\infty}\leq M\right\}$.

\begin{theorem}\label{mainthm}
Fix $s>1$, $\gamma <3$, and $M>0$. If $\omega \in \mathbb{R}^2_{+}$ is irrational, then  there exists $\varepsilon_{s, M,\omega, \gamma}>0$  such that for  all $\psi_0 \in C^{\infty}(\mathbb{T}_{\omega}^2)$ with $\|\psi_0\|=\varepsilon<\varepsilon_{s, M, \omega, \gamma}$ and  $\mbox{supp } \hat{\psi}_0 \subset Q_M$, \eqref{irrNLS} has a unique solution $\psi \in C([0,1/\varepsilon^2]; H^{s}(\mathbb{T}^2_{\omega}))$ such that
	\begin{align*}
		|\hat{\psi}_j(t)| < \varepsilon^{\gamma}
	\end{align*}
for all $t \in [0,1/\varepsilon^2]$ when $j \not\in Q_M$.

%
\end{theorem}

\begin{remark}\label{CF} This result should be compared to Theorem 1.1 in \cite{CF12},  where the authors consider \eqref{NLS} on a square torus and study the evolution  of a certain  initial data with zero high modes, and prove that  instead certain high modes of the solution  are significantly large in the same time frame as our Theorem \ref{mainthm}. Specifically, they show that for a very specific solution localized in $Q_1$ and $\eta \in (0,1)$,  $|\hat{\psi}_j(t)|>\varepsilon^{2\eta+1}$ when $t \sim \varepsilon^{-2(1-\eta|j|^{-2})}$, $|j| \lesssim (\log \varepsilon^{-1})^{1/4}$, and $j$ belongs to a certain subset of $j\mathbb{Z}^2$.  As the authors remark in \cite{CF12},  theirs  is an  example of {\it energy cascade} in the sense of \cite{C006}. Theorem  \ref{mainthm} above shows that  such a cascade is prevented when the torus is irrational. As we will demonstrate in Section \ref{sec.resonances}, this is mainly due to the fact that  in the irrational case the resonant set $\mathcal{R}$  in $\mathbb{Z}^2$ decouples as  $\mathcal{R}= \mathcal{R}_1 \cap \mathcal{R}_2$, where $\mathcal{R}_i, \, \, i=1,2,$ is  described  via a relationship involving only the $i$-th coordinates  of the vectors in $\mathcal{R}$. This suggests a decoupling mechanism that may set the 2D NLSE  we consider here in a framework that may be studied using inverse scattering theory, usually implemented for the cubic 1D NLSE, which is integrable.
\end{remark}

In order to prove Theorem \ref{mainthm}, we use $\omega$ as a perturbative parameter that changes the spectrum ($\lambda_k$) of the Laplace-Beltrami operator and allows us to control the resonant parts of the Hamiltonian.  We will compare the leftover resonances to the resonances in the case of the standard (rational) torus and show that the same type of energy cascades that occur there are not possible in our case, see Remark \ref{CF} above and the Appendix.   A result which is orthogonal to the one we prove here, but that also involves changing the spectrum of the Laplace operator can be found in  \cite{B96}. In fact there Bourgain shows that for fixed $m, \, s>1$, one can modify the Laplace operator of a nonlinear wave equation   to obtain a solution which Sobolev norm $H^s$ grows in time as $|t|^m$.

The blueprint of this paper starts in Section \ref{sec.primer} with a finite dimensional introduction to the most technical aspect of the argument: Birkhoff Normal Forms. The following section, Section \ref{sec.birkhoff},  extends the finite variable setting Birkhoff Normal Form transformation to the infinite dimensional Hamiltonian defined above in  \eqref{ham}. The beginning of this section is essentially the beginning of the proof of Theorem \ref{mainthm}, so it is useful to recall that the regularity and  the bound of the Fourier support are fixed. 
  After the infinite variable normal form section, we examine the resonant structure of the {\it normalized} Hamiltonian in Section \ref{sec.resonances}. The structure of the normalized Hamiltonian drives the dynamics of the system as shown in Section \ref{sec.dynamics}.  The final section, Section \ref{sec.proof}, contains the proof of Theorem \ref{mainthm}. The paper has also an appendix, Appendix \ref{app.A}, where  the result of this paper is compared with that in \cite{CKSTT10} and \cite{GK12}.

The contents of Sections \ref{sec.primer} and \ref{sec.birkhoff} can also  be found in the references \cite{Wig90}, \cite{CF12}, and \cite{BG}. Furthermore, Lemmas \ref{lemmaT} and \ref{nocasc} of Section \ref{sec.dynamics} have analogous  statements in \cite{CF12}.  Details are provided in this paper for completeness, and in order to serve as a reference for future work concerning a possible integrable structure for this 2D NLS model, similar to that of the  1D NLS, and as mentioned  in Remark \ref{CF}. 

	\section{Primer on Birkhoff Normal Forms}\label{sec.primer}

	Birkhoff Normal Forms is an infinite-dimensional normal form method derived from the normal form method for finite dimensional Hamiltonian systems. In this section we work out an example that should serve as a gentle introduction to the next section, in which we consider an infinite dimensional situation corresponding to the initial value problem \eqref{irrNLS}. The material in this section is standard and can be found in Chapter 19 of \cite{Wig90}.

Consider the symplectic manifold $(\mathcal{M}, \Omega)$ where $\mathcal{M} \cong \mathbb{C}^3 \hookrightarrow \mathbb{C}^6$ and $z \in \mathcal{M}$ is of the form 
	\begin{align*}
		z = \left(\begin{array}{c} z_1 \\ z_2 \\ z_3 \\ \bar{z}_1 \\ \bar{z}_2 \\ \bar{z}_3 
			\end{array} \right)
	\end{align*}
and the symplectic form $\Omega$ is the 2-form defined by 
	\begin{align*}
		\Omega(z, y) = \langle z, J y\rangle= i\sum_{j=1}^3 z_i \bar{y}_i -\bar{z}_iy_i
	\end{align*}
where $\langle \cdot, \cdot \rangle$ is a real dot-product and 
	\begin{align*}
		J= i \left( \begin{array}{cc} 0 & I_3\\ -I_3 & 0 \end{array} \right)	\quad\mbox{with } \quad 
		I_3= \left( \begin{array}{ccc} 1 & 0&0\\ 0&1 & 0\\ 0&0& 1 \end{array} \right).
	\end{align*}
For any two functions $f, g$ defined on $\mathcal{M}$ we define the Poisson bracket $\{\cdot, \cdot \}$ by
	\begin{align}\label{poisson}
		\{f, g\}:=  \Omega( \nabla_{z}f, \nabla_{z}g)= i\sum_{j=1}^3 \frac{\partial f}{\partial z_j} \frac{\partial g}{\partial \bar{z}_j} -\frac{\partial f}{\partial \bar{z}_j} \frac{\partial g}{\partial z_j}.
	\end{align}
Note that if $f$ is a $k$ degree homogeneous polynomial and $g$ is a $m$ degree homogeneous polynomial, then $\{f,g\}$ is a $k+m-2$ degree homogeneous polynomial.  

Consider the Hamiltonian 
	\begin{align}  \label{Hamtrunc}
		H(z)= \lambda_1|z_1|^2+\lambda_2|z_2|^2+\lambda_3|z_3|^2+ \sum_{\alpha_1+\alpha_2=\alpha_3+\alpha_4; \, \,  \alpha_i\in {1,2,3}} z_{\alpha_1}z_{\alpha_2}\bar{z}_{\alpha_3}\bar{z}_{\alpha_4},
	\end{align}
where $\lambda_i \in \mathbb{R}\setminus \{0\}$. The corresponding ODE is 
	\begin{align}\label{ode}
		\partial_t z = J\nabla_{z}H(z)=:X_H(z).
	\end{align}

We will perform a fourth order normalization of $H$.  Let $G_4$ be a general fourth order real-valued homogeneous polynomial in the variables $z_1, z_2,...,\bar{z}_3$
	\begin{align} \label{auxham}
		G_4(z)= \sum_{\beta \in \{\pm 1,\pm 2,\pm 3\}^4} g_{\beta} z_{\beta},
	\end{align}
where $z_{\beta}=\prod_{\beta_i>0}z_{\beta_i}\prod_{\beta_i<0}\bar{z}_{|\beta_i|}$. 

 Let $F$ be a polynomial in $z$ and define the adjoint function with respect to another polynomial $G$ by
	\begin{align}\label{ad}
		ad_G(F)= \{F, G\}.
	\end{align}
Moreover, let $ad_G^j= ad_G \circ ad^{j-1}_G$ and 
	\begin{align*}
		\exp(ad_G)=I+ad_G+\frac{1}{2!} ad^2_G+ \frac{1}{3!}ad^3_G+ \cdots
	\end{align*}
Putting $H$ into normal form will consist of plugging $H$ into the function $\exp(ad_{G_4})$ and choosing the values of $g_{\alpha}$ (from \eqref{auxham}) in order to simplify the new polynomial.  Why $\exp(\mbox{ad}_{G_4})H$ should still be a Hamiltonian and dynamically similar to $H$ will be shown later.  

With 
	\begin{align*}
		 H_0(z)=\lambda_1|z_1|^2+\lambda_2|z_2|^2+\lambda_3|z_3|^2 \hspace{.2cm} 
\mbox{ and }
		\hspace{.2cm} P(z)= \sum_{\alpha_1+\alpha_2=\alpha_3+\alpha_4} z_{\alpha_1}z_{\alpha_2}\bar{z}_{\alpha_3}\bar{z}_{\alpha_4}
	\end{align*} 
and $H$ as in \eqref{Hamtrunc},  we consider
	\begin{align}\label{expmap}
		\exp(ad_{G_4})H&=\exp(ad_{G_4})(H_0+P)\\
			&=(I+ad_{G_4}+\tfrac{1}{2!} ad^2_{G_4}+ \tfrac{1}{3!}ad^3_{G_4}+ \cdots)(H_0+P) \nonumber\\	
							&=H_0 +(P+ad_{G_4}H_0)+(ad_{G_4}P+\tfrac{1}{2!}ad^2_{G_4}H_0)\nonumber\\
			&\hspace{.5cm}+(\tfrac{1}{2!}ad^2_{G_4}P+\tfrac{1}{3!}ad^3_{G_4}H_0)+\cdots\nonumber
	\end{align}
and we see that each $ad_{G_4}^{k-1}P+ad_{G_4}^{k}H_0$ is of order $2(k+1)$ for $k=1, 2, 3,...$.  Since we are performing  a normal form transformation at order 4 we will focus on the term
	\begin{align} \label{homoeqn}
		P+ad_{G_4}H_0 = P + \{H_0, G_4\}.
	\end{align}
A reasonable initial goal would be to completely eliminate the term $P+ad_{G_4}H_0$ by our choice of the coefficients $g_{\beta}$.  We will see that the ability to eliminate any monomial $z_{\beta}$ will depend on the interplay between $\beta$ and $(\lambda_1, \lambda_2, \lambda_3)$.  Now for some computations, using  \eqref{ad} and \eqref{poisson}:
	\begin{align*}
		P+ad_{G_4}H_0 &=P+ \{H_0, G_4\}\\&= \sum_{\alpha_1+\alpha_2=\alpha_3+\alpha_4} z_{\alpha_1}z_{\alpha_2}\bar{z}_{\alpha_3}\bar{z}_{\alpha_4}+ \Omega( \nabla_{z,\bar{z}}H_0, \nabla_{z, \bar{z}}G_4)\\
			&= \sum_{\alpha_1+\alpha_2=\alpha_3+\alpha_4} z_{\alpha_1}z_{\alpha_2}\bar{z}_{\alpha_3}\bar{z}_{\alpha_4}  +  i\sum_{j=1}^3 \frac{\partial \lambda_j|z_j|^2}{\partial z_j} \frac{\partial G_4}{\partial \bar{z}_j} -\frac{\partial \lambda_j|z_j|^2}{\partial \bar{z}_j} \frac{\partial G_4}{\partial z_j}\\
			&= \sum_{\alpha_1+\alpha_2=\alpha_3+\alpha_4} z_{\alpha_1}z_{\alpha_2}\bar{z}_{\alpha_3}\bar{z}_{\alpha_4}  +  \sum_{\beta \in \{\pm 1,\pm 2,\pm 3\}^4} g_{\beta}  i\sum_{j=1}^3 \frac{\partial \lambda_j|z_j|^2}{\partial z_j} \frac{\partial z_{\beta}}{\partial \bar{z}_j} -\frac{\partial \lambda_j|z_j|^2}{\partial \bar{z}_j} \frac{\partial z_{\beta}}{\partial z_j}.
	\end{align*}
Then for each monomial, $z_{\beta}$ with $\beta=(\beta_1, \beta_2, \beta_3,\beta_4)$, we have the identity
	\begin{align*}
		\sum_{j=1}^3 \frac{\partial \lambda_j|z_j|^2}{\partial z_j} \frac{\partial z_{\beta}}{\partial \bar{z}_j} -\frac{\partial \lambda_j|z_j|^2}{\partial \bar{z}_j} \frac{\partial z_{\beta}}{\partial z_j}=\left(\sum_{\beta_i>0} \lambda_{\beta_i}-\sum_{\beta_i<0} \lambda_{|\beta_i|}\right)z_{\beta}.
	\end{align*}
Thus,
	\begin{align*}
		\sum_{\alpha_1+\alpha_2=\alpha_3+\alpha_4} z_{\alpha_1}z_{\alpha_2}\bar{z}_{\alpha_3}\bar{z}_{\alpha_4}  +  \sum_{\beta \in \{\pm 1,\pm 2,\pm 3\}^4} g_{\beta}  i\sum_{j=1}^3 \frac{\partial \lambda_j|z_j|^2}{\partial z_j} \frac{\partial z_{\beta}}{\partial \bar{z}_j} -\frac{\partial \lambda_j|z_j|^2}{\partial \bar{z}_j} \frac{\partial z_{\beta}}{\partial z_j}
\\=\sum_{\alpha_1+\alpha_2=\alpha_3+\alpha_4} z_{\alpha_1}z_{\alpha_2}\bar{z}_{\alpha_3}\bar{z}_{\alpha_4}  +  \sum_{\beta \in \{\pm 1,\pm 2,\pm 3\}^4} g_{\beta}  i\big(\sum_{\beta_i>0} \lambda_{\beta_i}-\sum_{\beta_i<0} \lambda_{|\beta_i|}\big)z_{\beta}.
	\end{align*}
If $\beta$ is such that $z_{\beta}$ does not appear as a term in $P$, i.e. $z_{\beta}$ is not of the form $z_{\alpha_1}z_{\alpha_2}\bar{z}_{\alpha_3}\bar{z}_{\alpha_4}$ where $\alpha_1+\alpha_2=\alpha_3+\alpha_4$, we let $g_{\beta}=0$.  Therefore,  
	\begin{align*}
		\sum_{\alpha_1+\alpha_2=\alpha_3+\alpha_4} z_{\alpha_1}z_{\alpha_2}\bar{z}_{\alpha_3}\bar{z}_{\alpha_4}  +  \sum_{\beta \in \{\pm 1,\pm 2,\pm 3\}^4} g_{\beta}  i\big(\sum_{\beta_i>0} \lambda_{\beta_i}-\sum_{\beta_i<0} \lambda_{|\beta_i|}\big)z_{\beta}\\
		=\sum_{\alpha_1+\alpha_2=\alpha_3+\alpha_4} z_{\alpha_1}z_{\alpha_2}\bar{z}_{\alpha_3}\bar{z}_{\alpha_4}  +  i g_{\alpha}  (\lambda_{\alpha_1}+\lambda_{\alpha_2}- \lambda_{\alpha_3}-\lambda_{\alpha_4})z_{\alpha_1}z_{\alpha_2}\bar{z}_{\alpha_3}\bar{z}_{\alpha_4}.
	\end{align*} 
	
By defining $g_{\alpha}= i (\lambda_{\alpha_1}+\lambda_{\alpha_2}- \lambda_{\alpha_3}-\lambda_{\alpha_4})^{-1}$ we can eliminate all fourth order terms from the new Hamiltonian.  However, $\lambda_{\alpha_1}+\lambda_{\alpha_2}- \lambda_{\alpha_3}-\lambda_{\alpha_4} \neq 0$ does not always hold which leads to the central difficulty of normal form theory.  Instances when $\lambda_{\alpha_1}+\lambda_{\alpha_2}- \lambda_{\alpha_3}-\lambda_{\alpha_4} = 0$ are known as {\bf\emph{resonances}}.  Working around these resonances makes up a large proportion of the literature on this subject.  

 For the moment, let's assume that the spectrum $\{\lambda_1, \lambda_2, \lambda_3\}$ is { \bf\emph{nonresonant at order four}}.  We will take this to mean that\footnote{This is equivalent to the standard definition of a nonresonance at order four satisfying the condition that for any $v=(v_1,v_2,v_3) \in \mathbb{Z}^3\setminus\{0\}$, with $|v_1|+|v_2|+|v_3|=4$, 
	\begin{align*}
		v_1\lambda_1+v_2\lambda_2+v_3\lambda_3 \neq 0.
	\end{align*}}
	\begin{align*}
		\lambda_{\alpha_1}+\lambda_{\alpha_2}- \lambda_{\alpha_3}-\lambda_{\alpha_4} \neq 0
	\end{align*}
whenever there aren't trivial  pairwise cancellations between the frequencies; i.e $\alpha_{1}=\alpha_{3}$ and $\alpha_2=\alpha_4$, or  $\alpha_{1}=\alpha_{4}$ and $\alpha_2=\alpha_3$.
%
Consequently, assuming nonresonance at order four implies that $(\alpha_1, \alpha_2, \alpha_3, \alpha_4)$ is a resonance if and only if  $\alpha_{1}=\alpha_{3}$ and $\alpha_2=\alpha_4$, or  $\alpha_{1}=\alpha_{4}$ and $\alpha_2=\alpha_3$. 

 Assuming $\{\lambda_1, \lambda_2, \lambda_3\}$ is nonresonant at order four, we let $g_{\alpha}= i (\lambda_{\alpha_1}+\lambda_{\alpha_2}- \lambda_{\alpha_3}-\lambda_{\alpha_4})^{-1}$ when $\lambda_{\alpha_1}+\lambda_{\alpha_2}- \lambda_{\alpha_3}-\lambda_{\alpha_4} \neq 0$ and we obtain
	\begin{align*}
		&\sum_{\alpha_1+\alpha_2=\alpha_3+\alpha_4} z_{\alpha_1}z_{\alpha_2}\bar{z}_{\alpha_3}\bar{z}_{\alpha_4}  +  i g_{\alpha}  (\lambda_{\alpha_1}+\lambda_{\alpha_2}- \lambda_{\alpha_3}-\lambda_{\alpha_4})z_{\alpha_1}z_{\alpha_2}\bar{z}_{\alpha_3}\bar{z}_{\alpha_4}\\
		&=\sum_{\alpha_1+\alpha_2=\alpha_3+\alpha_4 \atop \lambda_{\alpha_1}+\lambda_{\alpha_2}- \lambda_{\alpha_3}-\lambda_{\alpha_4} \neq 0} z_{\alpha_1}z_{\alpha_2}\bar{z}_{\alpha_3}\bar{z}_{\alpha_4}  -z_{\alpha_1}z_{\alpha_2}\bar{z}_{\alpha_3}\bar{z}_{\alpha_4}+ \sum_{\alpha_1+\alpha_2=\alpha_3+\alpha_4 \atop \alpha_{1}=\alpha_{3}, \alpha_2=\alpha_4}z_{\alpha_1}z_{\alpha_2}\bar{z}_{\alpha_3}\bar{z}_{\alpha_4}.
	\end{align*}
Therefore,
	\begin{align*}
		P+ \{H_0, G_4\}&= \sum_{\alpha_1, \alpha_2} |z_{\alpha_1}z_{\alpha_2}|^2 \sim \left(\sum_{i=1}^3 |z_{i}|^2 \right)^2.
	\end{align*}

Our new Hamiltonian is 
	\begin{align*}
		\exp(ad_{G_4})H&=H_0+ \left(\sum_{i=1}^3 |z_{i}|^2 \right)^2+ F_6+ F_8+\cdots\\
				&=H_0+ \|z\|^4+ F_6+ F_8+\cdots,
	\end{align*}
where $F_{2(k+1)}=ad_{G_4}^{k-1}P+ad_{G_4}^{k}H_0$.  We note that if we were to be working with the equivalent Wick ordered version of \eqref{Hamtrunc},  see for example \cite{BGibbs},
	\begin{align*}
		H(z)= \lambda_1|z_1|^2+\lambda_2|z_2|^2+\lambda_3|z_3|^2+ \sum_{\alpha_1+\alpha_2=\alpha_3+\alpha_4} z_{\alpha_1}z_{\alpha_2}\bar{z}_{\alpha_3}\bar{z}_{\alpha_4} - \|z\|^4
	\end{align*}
we can completely normalize the fourth degree monomial.

	\subsubsection{The Change of Variables is a symplectomorphism}\label{subsubsec.variables}
Now, we face the details of the symplectic transformation that is $\exp(ad_{G_4})$.
	\begin{lemma} Let $G$ be a Hamiltonian defined on $\mathcal{M}$, and let $\phi_t:\mathcal{M} \rightarrow \mathcal{M}$ be the time $t$ flow map for $G$ as in \eqref{ode}.
		Let $F:\mathcal{M} \rightarrow \mathbb{R}$ be a regular function.  Then 
			\begin{align*}
				\tfrac{d}{dt}(F \circ \phi_t) = \{F, G\} \circ \phi_t.
			\end{align*}
	\end{lemma}
	\begin{proof}
		\begin{align*}
			\frac{d}{dt}( F \circ \phi_t )(z) &= \nabla F(\phi_t(z)) \cdot \frac{d}{dt}\phi_t(z)\\
				&= \nabla F(\phi_t(z)) \cdot J\nabla G(\phi_t(z))\\
				&= \{F, G\}(\phi_t(z)).
		\end{align*}
	\end{proof}

	\begin{lemma}
		Let $G$ be a Hamiltonian defined on $\mathcal{M}$, and let $\phi_t:\mathcal{M} \rightarrow \mathcal{M}$ be the time $t$ flow map for $G$ as in \eqref{ode}.  Then 
			\begin{align*}
				\exp(ad_G)F=F \circ \phi_1
			\end{align*}
for any function $F:\mathcal{M} \rightarrow \mathbb{R}$.
	\end{lemma}
	\begin{proof}
		Let $g(t)= F \circ \phi_t$. Then $g(t) = g(0) + g'(0)t+\frac{1}{2!}g''(0)t^2+\cdots$ and by a symplectic identity
			\begin{align*}
				g'(t) = \frac{d}{dt} (F\circ \phi_t) =  \{F, G\} \circ \phi_t = ad_G(F) \circ \phi_t
			\end{align*}
		and thus 
			\begin{align*}
				g(1)= Id+ad_G(F)+ \tfrac{1}{2!}ad^2_G(F)+ \cdots=\exp(ad_G)(F),
			\end{align*}
 		since the dynamics of Hamiltonians preserve the symplectic structure, so does $\exp(ad_G)(F)$.
	\end{proof}

%
%
%

	\section{Birkhoff Normal Forms}\label{sec.birkhoff}
	
	 In this section, we return to our original infinite dimensional Hamiltonian \eqref{ham}
	\begin{align*}
		H( \hat{\psi}) &= \tfrac{1}{2} \sum_{k \in \mathbb{Z}^d} \lambda_k |\hat{\psi}_k|^2 + \tfrac{1}{4}\sum_{ k_1+k_2= h_1+h_2} \hat{\psi}_{k_1} \hat{\psi}_{k_2}\bar{\hat{\psi}}_{h_1}\bar{\hat{\psi}}_{h_2}\\
				&=:H_0 + P.
	\end{align*}
  Assume that $\omega= (\omega_1, \omega_2) \in \mathbb{R}^2_{+}$ is irrational ($\omega \cdot m \neq 0, \forall m \in \mathbb{Z}^2$).   Then, for any fixed $N>0$, the subset of frequencies  $(\lambda_k)_{k \in \mathbb{Z}^2, |k|\leq N}=(\omega_1|k^{(1)}|^2+\omega_2|k^{(2)}|^2)_{k \in \mathbb{Z}^2}$ satisfies the following nonresonance condition:

	\begin{align}\label{nonres}
		 \lambda_{k_1}+\lambda_{k_2}-\lambda_{k_3}-\lambda_{k_4} \neq 0.
	\end{align}
We  can thus characterize the resonances $(k_1, k_2, k_3, k_4)$ satisfying
	\begin{align*}
		 \lambda_{k_1}+\lambda_{k_2}-\lambda_{k_3}-\lambda_{k_4} = 0
	\end{align*}
by the condition
	\begin{align*}
		|k^{(1)}_{1}|^2+|k^{(1)}_{2}|^2-|k^{(1)}_{3}|^2-|k^{(1)}_{4}|^2 &= 0 \\\mbox{ and } \quad |k^{(2)}_{1}|^2+|k^{(2)}_{2}|^2-|k^{(2)}_{3}|^2-|k^{(2)}_{4}|^2  &= 0.
	\end{align*}

	We will also consider what we will call ``weak'' resonances:  4-tuples, $(k_1, k_2, k_3, k_4)$, satisfying
	\begin{align*}
		 |\lambda_{k_1}+\lambda_{k_2}-\lambda_{k_3}-\lambda_{k_4}| \leq 1
	\end{align*}
	
	We will go back to this resonant sets later, now we need to establish the functional setting,  with the frequency space definition of the $L^2$ Sobolev norm.

\begin{definition}
For $x=\{x_n\}_{n \in \mathbb{Z}^d}$ and $s \in (0, \infty)$, define the standard Sobolev norm as
\begin{align*}
\|x\|_s := \sqrt{ \sum_{n\in \mathbb{Z}^d}|x_n|^2 \langle n \rangle^{2s}   } 
\end{align*}
where $\langle n \rangle:= \sqrt{|n|^2+1}$.  Define $h^s(\mathbb{Z}^d)$ as
\begin{align*}
h^s:=h^s(\mathbb{Z}^d):= \left\{ x= \{x_n\}_{n\in \mathbb{Z}^d} \,: \, \|x\|_s<\infty  \right\}.
\end{align*}
\end{definition}

Furthermore, for $R>0$, $s \in (0, \infty]$, define
	\begin{align*}
		B_s(R):=\{x\in h^s:\|x\|_s \leq R\}.
	\end{align*}
	
Generalizing equation \eqref{ode}, for any Hamiltonian $H$, we let its corresponding Hamiltonian vector field be defined as 
	\begin{align*}
		X_H(z):=J \left( \begin{array}{c} \partial_{z}H(z)\\  \partial_{\bar{z}} H(z) \end{array} \right)
	\end{align*}
where
	\begin{align*}
		J= i \left( \begin{array}{cc} 0 & Id\\ -Id & 0 \end{array} \right).
	\end{align*}

 The following is a well-known lemma that appears throughout the literature in regards to normal forms (Lemma \ref{ident} appears as Remark 4.15 in \cite{BG}.).  We need this result to ensure that our normal form change of variables preserves dynamical properties of the original Hamiltonian system.  Lemma \ref{ident} is used to justify inequality \eqref{identest} in the Birkhoff Normal Form lemma below, and follows directly from Duhamel's formula. 
\begin{lemma}\label{ident}
		Let $\chi$ be an analytic function with Hamiltonian vector field $X_\chi$ which is analytic as a map from $B_s(R)$ to $h^s$; fix $C < R$.  Assume that $\sup_{\|z\|_s \leq R}\|X_{\chi}(z)\|_{s} < C$, and consider the time $t$ flow $T^t$ of $X_{\chi}$.  Then, for $|t|\leq 1$, one has 
			\begin{align*}
				\sup_{\|z\|_s \leq R-C} \|T^t(z)-z\|_s \leq \sup_{\|z\|_s \leq R}\|X_{\chi}(z)\|_{s}.
			\end{align*}
	\end{lemma}

Now we define a cut-off for $z \in h^s$, $z=(z_k)_{k\in \mathbb{Z}^d}$, at height $N$.
For a positive integer $N$, let 
	\begin{align*}
		z^{l}:= \left\{ \begin{array}{lc} z_k & |k|<N\\ 0& |k|\geq N\end{array}\right.\\
		z^h:= \left\{ \begin{array}{lc} z_k & |k|\geq N\\ 0& |k|< N.\end{array}\right.
	\end{align*}
	
	We finally have all the elements to present the infinite dimensional analogue of the Birkhoff Normal Form transformation.  Other infinite-dimensional analogues that require stronger nonresonance conditions on the frequencies $(\lambda_k)_{k \in \mathbb{Z}^2}$ can be found in \cite{BG}.
	\begin{lemma}[Birkhoff Normal Form of Order 4] \label{iterlemma}
Consider the Hamiltonian $H=H_0+P$ defined in \eqref{ham}. Fix $s \in (0, \infty]$.  There exists a positive $R_0\ll1$ such that  for any $N>1$, there exists $A>1$ and an analytic canonical transformation
\begin{align*}
\mathscr{T}: B_s\Big( \tfrac{R_0}{A}\Big) \rightarrow h^s
\end{align*}
which transforms $H$ into
\begin{align*}
\tilde{H}:= H \circ \mathscr{T}= H_0 + \mathscr{L}+ \mathscr{U}+\mathscr{R} .
\end{align*}
 For any $R < R_0A^{-1}$, there exists a constant $C_{s, N,\omega}$ such that the following properties are fulfilled:
\begin{enumerate}
\item  The transformation $\mathscr{T}$ satisfies
	\begin{align}\label{identest}
	\sup_{z \in B_s(R)} \|z-\mathscr{T}(z)\|_s \leq C_{s, N, \omega}R^3;
	\end{align}   
\item $\mathscr{L}$ is a homogeneous polynomial of degree $4$; it is resonant (in normal form, $\{H_0,\mathscr{L}\}=0$) ; $\mathscr{U}$ is a homogeneous polynomial of degree 4 and has a zero of order 1 at the origin in terms of the $z^h$ variable; $\mathscr{R}$ has a zero of order 6 at the origin.
\end{enumerate}

Finally, the canonical transformation is a symplectomorphism from $B_s\Big( \tfrac{R_0}{A}\Big)$ into a neighborhood of the origin of $h^s$ for which the same estimate, \eqref{identest}, is fulfilled by the inverse canonical transformation.
\end{lemma}

\begin{proof}


Consider the Taylor expansion of $P$ in terms of the $z^h$ variables:
	\begin{align*}
		P= P_0+P_1+ P_2+P_3 + P_4.
	\end{align*}
Note that $P_0$ is defined only in terms of $z^{l}$, and thus $P_0$ is a polynomial in only finitely many variables. The real-valued auxiliary function generating the change of variables will be a sum of two polynomials: one polynomial in finitely many variables designed to eliminate all resonant tuples and one polynomial in infinitely many variables meant to remove all non-``weak'' resonant tuples. If the auxiliary function were to be a polynomial of infinitely many variables and eliminated all resonances, it would be impossible to establish the convergence of $\mathscr{T}$ without a strong nonresonance condition on the linear spectrum $(\lambda_k)_{k \in \mathbb{Z}^2}$ Therefore, we have established the weak form of resonance for the infinite variable polynomial. Now consider the truncated Hamiltonian
	\begin{align}\label{truncated}
		H_0+P_0.
	\end{align}

As typical, we are looking for a Lie transform, $\mathscr{T}$, that eliminates the nonresonant part of order $4$ in expression \eqref{truncated}.  Let $\chi=\chi_1+\chi_2$ be the homogeneous degree $4$, real-valued polynomial generating $\mathscr{T}$.  Then, as in the finite variable case \eqref{expmap}, we obtain:
	\begin{align*}
		(H_0+P) \circ \mathscr{T}= H_0&+\{\chi_1+\chi_2, H_0\}  +P_0+P_1+P_2+P_3+P_4\\
			&+\sum_{m\geq 1} \frac{1}{m!}\ad^m_{\chi}(P)+\sum_{m\geq 2} \frac{1}{m!}\ad^m_{\chi}(H_0)
	\end{align*}
and the  order four terms of $(H_0+P) \circ \mathscr{T}$ as $\{\chi, H_0\} + P$.  As in the finite dimensional case, we choose $\chi_1$ so that the only summands in the homogeneous polynomial $\{\chi_1, H_0\} +P_0$ correspond to resonant tuples of indexes.  We define
	\begin{align*}
		\mathscr{L}&:=\{\chi_1, H_0\} + P_0
	\end{align*}
	Also, we choose $\chi_2$ so that the only summands in the homogeneous polynomial $\{\chi_2, H_0\} +P_1+P_2+P_3+P_4$ correspond to ``weakly'' resonant tuples of indexes. We define
	\begin{align*}
		\mathscr{U}&:=\{\chi_2, H_0\} +P_1+P_2+P_3+P_4
	\end{align*}
We will use the same Lie transform for the complete Hamiltonian $H$:
	\begin{align*}
		H\circ \mathscr{T}=(H_0&+P_0+ P_1+\cdots + P_4) \circ \mathscr{T}\\
			= H_0&+\mathscr{L}+\mathscr{U}+\sum_{m\geq 1} \frac{1}{m!}\ad^m_{\chi}(P)+\sum_{m\geq 2} \frac{1}{m!}\ad^m_{\chi}(H_0).
	\end{align*}
Finally, we organize the remaining parts of the new Hamiltonian as
	\begin{align*}
		\mathscr{R} &:= \sum_{m\geq 1} \frac{1}{m!} \ad^m_{\chi}(P)+\sum_{m\geq 2} \frac{1}{m!} \ad^m_{\chi_{\ell}}(H_0).
	\end{align*}
This completes the formal argument of the proof.  

 The constant $A>1$, from the statement,  is the distortion of the Lie transform given by first computing the supremum of all divisors
	\begin{align*}
		C_{N, \omega} := \max\left( \sup \frac{1}{|\lambda_{k_1}+\lambda_{k_2}-\lambda_{k_3}-\lambda_{k_4}|}, 1\right)
	\end{align*}
where the supremum is taken over all $(k_1, k_2, k_3, k_4)$ that correspond to monomials in $\chi(z) = \sum z_{k_1}z_{k_2}\bar{z}_{k_3}\bar{z}_{k_4}$.  These 4-tuples are exactly those satisfying $|k_i|\leq N$ and are nonresonant in the sense of \eqref{nonres}. Therefore, there are only finitely many such 4-tuples, and we can ensure that $C_{N, \omega}< \infty$. $C_{N, \omega}$ may grow incredibly fast with respect to $N$, but $N$ is fixed at the beginning of our arguments. The boundedness of $C_{N, \omega}$ implies that 
	\begin{align} \label{bound}
		\|X_{\chi}(z)\|_s \leq C_s C_{N, \omega} \|z\|_s^3\leq C_{s, N, \omega} \|z\|^3_s.
	\end{align}
Inequality \eqref{bound}  holds for all $s>0$ because $\chi$ is a polynomial in only $z^l$. Denote $C:=C_{s, N, \omega}$.  Now we show that the initial value problem 
	\begin{align*}
		\left\{ \begin{array}{ll} &\partial_t z= X_{\chi}(z)=J\nabla_{z, \bar{z}}\chi(z),  \\
						&z(0) \in B_s(\tfrac{R}{C^{1/2}}) \subset h^s  \end{array} \right.
	\end{align*}
is well-posed up until $t \lesssim  R^{-2}$. In fact if we consider the operator 
$$\Phi(z(t)):=z(0) +\int_0^tX_{\chi}(z(s))ds,$$
and we use \eqref{bound}, we have that 
$$\|\Phi(z(t))\|_{s}\lesssim \frac{R}{C^{1/2}} +t \|X_{\chi}(z)\|_s\lesssim  \frac{R}{C^{1/2}} +t C \|z\|^3_s$$
and hence  for $t\lesssim  R^{-2}$ the operator $\Phi$ sends a ball of radius $4\frac{R}{C^{1/2}} $ into itself. In a similar manner one can show that for the same interval of time $\Phi$ is also a contraction.
 Thus, since $\mathscr{T}$ is the time 1 flow map for $\chi$, $\mathscr{T}$ converges for $R$ small enough. We then let $A := C^{1/2}$.

Estimate \eqref{identest} follows from Lemma \ref{ident} and the fact that $P$ has a zero of order four at the origin.

\end{proof}

\begin{remark}
At this point, we must discuss what one should expect if one continues the normal form reductions for increasingly higher order terms.  Inspection reveals that the order six term of $\tilde{H}$   is similar to the order six term for the Hamiltonian for the quintic NLSE, namely $\tfrac{1}{2}\int |\nabla \psi|^2+ \tfrac{1}{6} \int |\psi|^{6}.$ This observation combined with the growth of Sobolev norms for solutions of the quintic NLSE in our current setting, demonstrated by Haus and Procesi \cite{HP14}, suggests that the expected cascading of energy, if at all present in this model, is driven by higher order effects.  
\end{remark}


	\section{Resonances}\label{sec.resonances}
After applying Lemma \ref{iterlemma} at height $N$, the Hamiltonian $H$ is transformed into $\tilde{H}$:
	\begin{align} \label{newNLS}
		\tilde{H}(\{z_k\})&= H_0 + \mathscr{L}+ \mathscr{U} + \mathscr{R} \nonumber\\
				&= \tfrac{1}{2}\sum_{k \in \mathbb{Z}^2} \lambda_k |z_k|^2+ \tfrac{1}{4} \sum_{(k_1, k_2, k_3, k_4) \in \mathcal{R}^N} z_{k_1}z_{k_2}\bar{z}_{k_3}\bar{z}_{k_4}+ \tfrac{1}{4} \sum^*_{|k_1|>N \ or \ |k_3|>N} z_{k_1}z_{k_2}\bar{z}_{k_3}\bar{z}_{k_4}+\mathscr{R},
	\end{align}
	where  $\sum^*$ is the sum taken over $(k_1, k_2, k_3, k_4)$ satisfying the conservation of momentum condition and the ``weak'' resonance condition:
	\begin{align} \label{momentum}
		k_1+k_2=k_3+k_4, \mbox{ and }\\
		|\lambda_{k_1}+\lambda_{k_2}=\lambda_{k_3}+\lambda_{k_4}|<1
	\end{align}
$\mathcal{R}$ is the set of resonances defined by
	\begin{align*}
		\mathcal{R}&:= \left\{ (k_1, k_2, k_3, k_4) \in (\mathbb{Z}^2)^4~:~ k_1+k_2=k_3+k_4\right\}\\
			&\hspace{.5cm} \bigcap \left\{ (k_1, k_2, k_3, k_4) \in (\mathbb{Z}^2)^4~:~ \lambda_{k_1}+\lambda_{k_2}=\lambda_{k_3}+\lambda_{k_4}\right\}.
	\end{align*}
and 
	\begin{align*}
		\mathcal{R}^N&:= \left\{ (k_1, k_2, k_3, k_4) \in \mathcal{R}~:~ |k_i|\leq N, i=1, 2, 3, 4\right\}.
	\end{align*}

For any irrational vector $\omega$, the resonance condition $ \lambda_{k_1} + \lambda_{k_2} =\lambda_{k_3}+ \lambda_{k_4}$ is equivalent to 
		\begin{align*}
			 \omega_1\left[\left(k^{(1)}_1\right)^2 +\left(k^{(1)}_2\right)^2 -\left(k^{(1)}_3\right)^2 - \left(k^{(1)}_4\right)^2\right]=
			-\omega_2\left[\left(k^{(2)}_1\right)^2 +\left(k^{(2)}_2\right)^2 -\left(k^{(2)}_3\right)^2 - \left(k^{(2)}_4\right)^2\right]
		\end{align*}
and thus is equivalent to the following two independent one-dimensional resonance equations
		\begin{align*}
			\left(k^{(1)}_1\right)^2 +\left(k^{(1)}_2\right)^2 =\left(k^{(1)}_3\right)^2 + \left(k^{(1)}_4\right)^2\\
			\left(k^{(2)}_1\right)^2 +\left(k^{(2)}_2\right)^2 =\left(k^{(2)}_3\right)^2 + \left(k^{(2)}_4\right)^2.
		\end{align*}
	By the definition of vector addition, the first resonance condition (conservation of momentum \eqref{momentum}) also decomposes into one-dimensional resonance equations.  
	Thus $\mathcal{R}= \mathcal{R}_1 \cap \mathcal{R}_2$, where at the end
		\begin{align*}
			\mathcal{R}_i := &\left\{ (k_1, k_2, k_3, k_4) \in (\mathbb{Z}^2)^4~:~ k^{(i)}_1+k^{(i)}_2=k^{(i)}_3+k^{(i)}_4\right\}\\
&\bigcap \left\{ (k_1, k_2, k_3, k_4) \in (\mathbb{Z}^2)^4~:~ \left(k^{(i)}_1\right)^2 +\left(k^{(i)}_2\right)^2 =\left(k^{(i)}_3\right)^2 + \left(k^{(i)}_4\right)^2\right\}.
		\end{align*}

The nice feature about the dimensional decomposition of $\mathcal{R}$ is that the one-dimensional resonances are rather simple.  In fact, we can state a simpler characterization with a basic arithmetic lemma\footnote{See also Section 2.2 in \cite{CDS10}.}:
	\begin{lemma}\label{char}
		For $i=1, 2$,
		\begin{align*}
			\mathcal{R}_i &= \left\{ (k_1, k_2, k_3, k_4) \in (\mathbb{Z}^2)^4~:~ k^{(i)}_1=k^{(i)}_3 \mbox{ and  } k^{(i)}_2=k^{(i)}_4 \right\} \\
				&\hspace{.5cm}\bigcup \left\{ (k_1, k_2, k_3, k_4) \in (\mathbb{Z}^2)^4~:~ k^{(i)}_1=k^{(i)}_4 \mbox{ and } k^{(i)}_2=k^{(i)}_3 \right\}
		\end{align*}
	\end{lemma}
	\begin{proof}
		It suffices to show the equivalence of the two forms of $\mathcal{R}_i$ by considering $a, b , c, d \in \mathbb{Z}$ and proving that if $a+b=c+d$ and $a^2+b^2=c^2+d^2$ then either $a=c$ and $b=d$ or $a=d$ and $b=c$. The converse implication is obvious.

	Assume $a \neq c$.  Otherwise, $0=a-c=d-b$ would follow, implying $d=b$, at which point we would be done.  The condition $ a+b=c+d$ is equivalent to $a-c=d-b$ and implies
		\begin{align*}
			a^2+b^2=c^2+d^2 \Leftrightarrow a^2-c^2=d^2-b^2 \Leftrightarrow (a+c)(a-c)=(d+b)(d-b) \Rightarrow a+c=d+b.
		\end{align*}
		Together $a-c=d-b$ and $a+c=d+b$ imply $a=d$ and $b=c$.
	\end{proof}

We can now define resonances with respect to a fixed index $k$ and make some observations.  For $k \in \mathbb{Z}^2$, let
	\begin{align*}
		\mathcal{R}(k)= \left\{ (k_1, k_2, k_3) \in (\mathbb{Z}^2)^3~:~ (k_1,k_2, k_3, k)\in \mathcal{R}\right\}.
	\end{align*}
A result of the characterization of Lemma \ref{char} is that 
	\begin{align} \label{resok}
		\mathcal{R}(k)= \bigcup_{a, b \in \mathbb{Z}} \left\{\begin{array}{rl} \left( (k^{(1)}, b), (a, k^{(2)}), (a,b)\right), & \left( (a, k^{(2)}), (k^{(1)}, b), (a,b)\right), \\   \left( (k^{(1)}, k^{(2)}), (a,b), (a, b) \right), &\left( (a,b),  (k^{(1)}, k^{(2)}), (a, b)\right) \end{array} \right\}.
	\end{align}
The important aspect to note is that for $k \in \mathbb{Z}^2$, every ordered triple $(k_1, k_2, k_3) \in \mathcal{R}(k)$ must contain one $k_i$ so that the first component of $k_i$ is equal to the first component of $k$ and one $k_j$ ($j$ and $i$ could be equal) such that the second component of $k_j$ is equal to the second component of $k$. For the rational torus, this is not true.  For example, the triple $\left((1,1), (-1, 1), (0,0) \right)$ belongs to  $\mathcal{R}((0,2))$ when $\lambda_k = |k|^2$.  
$$
\begin{array}{c} \mbox{Figure 1}\\
\includegraphics[scale=0.5]{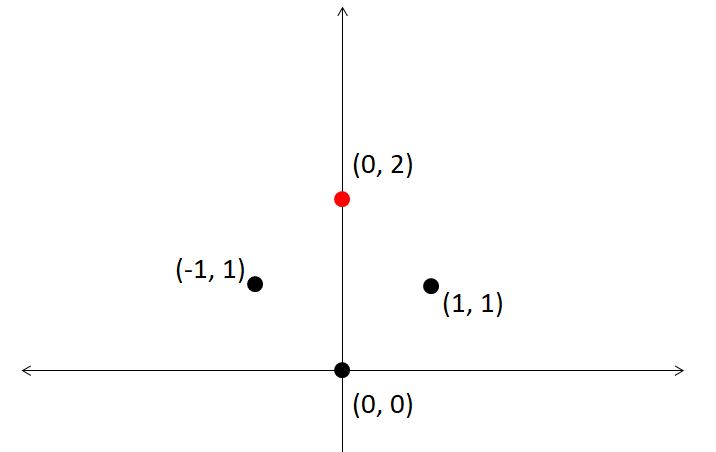}
\end{array}
$$

 In general, when $\omega$ is rational (i.e. $\omega\cdot m=0$ for some $m\in \mathbb{Z}^d$), there are three types of resonant 4-tuples. Resonant 4-tuples of any type form the vertices of a parallelogram in $\mathbb{Z}^2$.  The first type is what we will call {\bf degenerate}.  A degenerate resonant 4-tuple is a 4-tuple in which terms repeat.  In other words, $(k_1, k_2, k_3, k_4)$ is degenerate if $(k_1, k_2, k_3, k_4)= (k_1, k_2, k_1, k_2)$ or $(k_1, k_2, k_3, k_4)= (k_1, k_2, k_2, k_1)$. Hence the parallelogram reduces to a segment.  The second type of resonant tuple will be called {\bf parallel} resonances.  Parallel resonances are those in which the four points in the tuple form an axis parallel rectangle in $\mathbb{Z}^2$ (e.g.  $((0,0), (2,1), (2,0),(0,1))$).  The final type of resonance will be called {\bf nonparallel}, which are resonances that form a parallelogram that is not parallel (as in Figure 1 and Figure 3) to the axes in $\mathbb{Z}^2$.

$$
\begin{array}{c} \mbox{Figure 2: Parallel Resonance, any choice of $\omega$} \\ 
\includegraphics[scale=0.45]{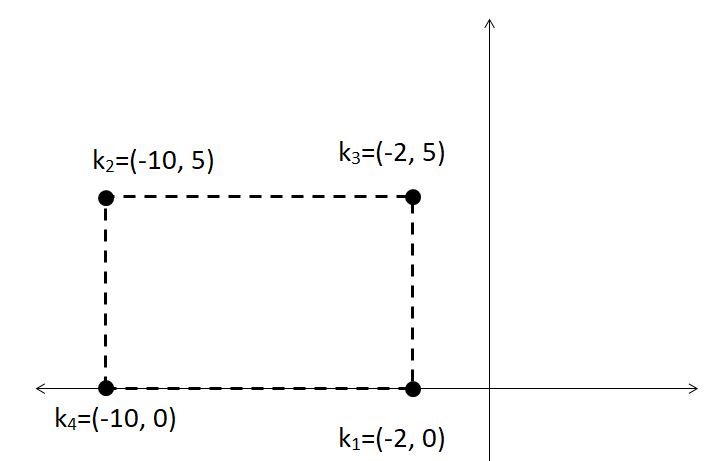}
\end{array}
$$

$$
\begin{array}{c}  \mbox{Figure 3: Nonparallel Resonance, $\omega=(1,2)$}\\\includegraphics[scale=0.45]{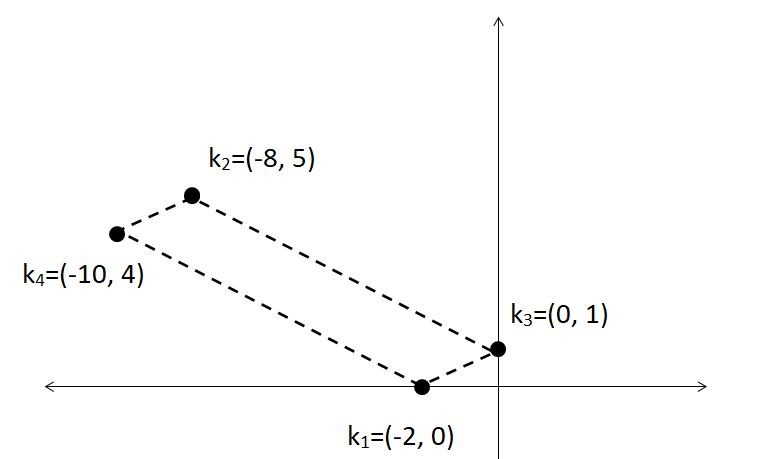}
\end{array}
$$

 When we declare $\omega$ to be irrational, the characterization \eqref{resok} shows us that there are no nonparallel resonances. The absence of nonparallel resonances is essential in distinguishing the difference between the dynamics of the rational and irrational case. In the rational case, there is a three-step process for mass to travel from one dyadic level to another:

$$
\begin{array}{cc} \mbox{Step 1} & \mbox{Step 2}  \\
\includegraphics[scale=0.3]{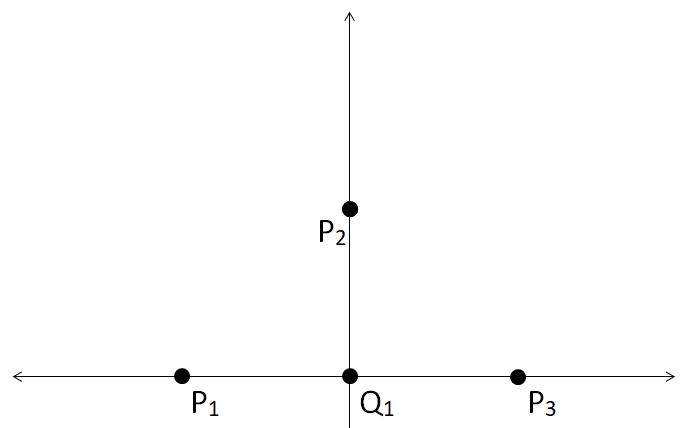} & \includegraphics[scale=0.3]{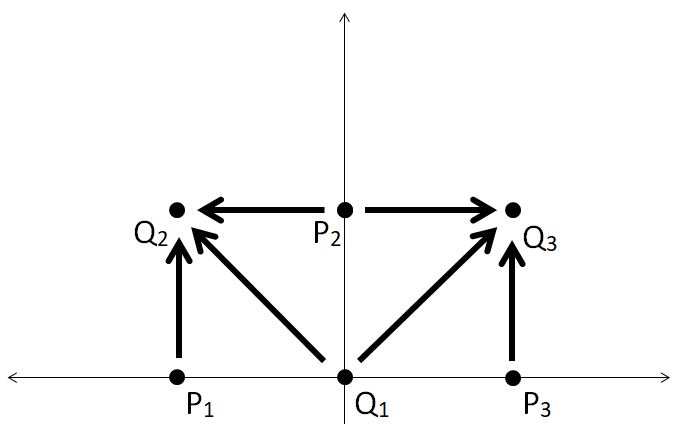}
\end{array}
$$

$$
\begin{array}{c} \mbox{Step 3}\\
 \includegraphics[scale=0.3]{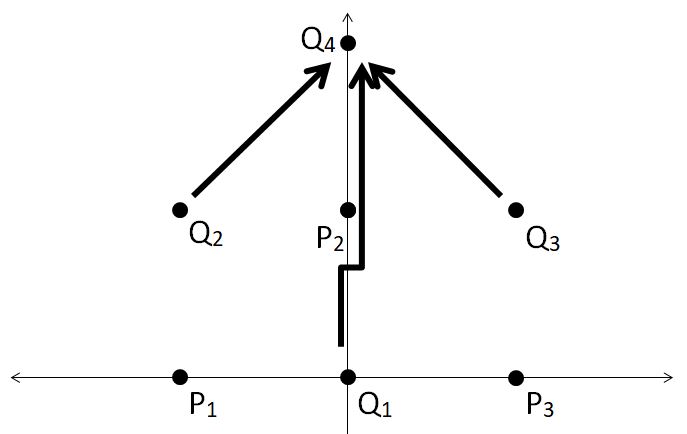}
\end{array}
$$

In Step 1, we start with an initial datum concentrated at the four points $P_1, P_2, P_3,$ and $Q_1$. Once time evolution begins, the resonant dynamics allow for mass to move from $P_1$, $P_2$, and $Q_1$ to newly activated $Q_2$. Simultaneously, mass moves from $Q_1$, $P_2$, and $P_3$ to $Q_3$ for the same reason. Now that $Q_2$ and $Q_3$ have mass, a new resonant 4-tuple appears in step 3, and as time evolves, mass moves from $Q_1$, $Q_2$, and $Q_3$ to $Q_4$.

 In the irrational case, Step 3 can not occur because it requires the appearance of a nonparallel resonant 4-tuple, namely $(Q_1, Q_2, Q_3, Q_4)$. 
 The absence of nonparallel resonant 4-tuples and the resulting dynamical consequences  will be detailed in the next section.  We  suspect that the difference in the dynamics between the rational and irrational cases is a result of the fact that the resonances decouple into products of one-dimensional resonances in the irrational case.  This may signify that  in some sense the irrational torus case inherits some  integrability  features of the one-dimensional cubic Schr\"odinger system, and as we know in this case the integrability manifested via conservation laws prevents any kind of cascade to high frequencies.

	\section{Dynamics}\label{sec.dynamics}

We will first study a truncation of $\tilde{H}$ (defined by  \eqref{newNLS}) at order 4. $H^*$ is commonly referred to as the {\bf Resonant System}:
	\begin{align*}
		H^* = H_0 + \mathscr{L} + \mathscr{U}
	\end{align*}
where $\mathscr{L}$ and $\mathscr{U}$ are given by Lemma \ref{iterlemma}. The corresponding system of equations is
	\begin{align}\label{gaugeNLS}
		i\dot{v}_k =\lambda_k v_k + \sum_{(k_1, k_2, k_3) \in \mathcal{R}^N(k)} v_{k_1}v_{k_2}\bar{v}_{k_3}+  \sum^*_{ |k_1|>N \ or \ |k_3|>N} v_{k_1}v_{k_2}\bar{v}_{k_3}.
	\end{align}
A gauge transformation allows us to reduce this system to 
	\begin{align}\label{truncnls}
		i\dot{u}_k = \sum_{(k_1, k_2, k_3) \in \mathcal{R}^N(k)} u_{k_1}u_{k_2}\bar{u}_{k_3}+  \sum^*_{|k_1|>N \ or \  |k_3|>N} u_{k_1}u_{k_2}\bar{u}_{k_3}e^{i(\lambda_{k_1}+\lambda_{k_2}-\lambda_{k_3}-\lambda_k)t}.
	\end{align}
We will first show that a solution to \eqref{truncnls} is  analytic as a function from $[0, T]$ to $H^{s}$, where $T$ is determined later.
In order to show analyticity an important observation is that $s>1$ implies $h^s$ is an algebra by Sobolev embedding.  Alternatively,  Young's inequality can be applied to $\mathscr{L}$ and $\mathscr{U}$,  allowing for the Taylor expansion in time of each $u_k$ which is the key to establishing the dynamic properties of system \eqref{truncnls}.  The following lemma is analogous to Lemma 2.3 in Carles and Faou \cite{CF12}. We include the full proof for completeness, but the proof for estimate \eqref{expbound} appears in full in \cite{CF12}.

\begin{lemma}\label{lemmaT}
Let $u(0) = \{u_k(0)\}_{k \in \mathbb{Z}^2} \in h^{s}$. There exists $T>0$ and a unique analytic in time solution $\{u_k\} : [0,T] \rightarrow h^s$ to \eqref{truncnls}. Moreover, there exist constants $C$ and $R$ such that for all $n \in \mathbb{N}$ and $\tau \leq T$,
	\begin{align} \label{expbound}
		\left|   \frac{d^nu_k}{(dt)^n}(\tau) \right| \leq CR^nn!.
	\end{align}
Moreover, $T$ can be taken to be $T = C\|u(0)\|^{-2}_s$, for some absolute constant $C>0$.
\end{lemma}

\begin{proof}
For the well-posedness  we look for the fixed point  in $C([0,T],h^s)$ of
	\begin{align*}
L(\{u_k\}):= \{u_k\}(0) -i\left\{\int_0^t\sum_{(k_1, k_2, k_3) \in \mathcal{R}^N(k)} u_{k_1}u_{k_2}\bar{u}_{k_3}+  \sum^*_{|k_1|>N \ or \  |k_3|>N} u_{k_1}u_{k_2}\bar{u}_{k_3}e^{i(\lambda_{k_1}+\lambda_{k_2}-\lambda_{k_3}-\lambda_k) \cdot }\right\}
	\end{align*}
and one can easily see, using the algebra structure of $h^{s}$ for $s>1$,  that
$$\|L(\{u_k\})\|_{h^{s}}\leq \|\{u_k(0)\}\|_{h^{s}}+TC\|\{u_k\}\|_{h^{s}}^3$$
so if $R\sim \|\{u_k(0)\}\|_{h^{s}}$ and $T\leq T_1:= C^{-1}\|\{u_k(0)\}\|_{h^{s}}^{-2}$ the mapping $L$ is a contraction and a fixed point, which is also solution, exists and is unique. 

Since $h^{s}$ for $s>1$ is an algebra, bootstrapping implies that $\{u_k\} \in C^{\infty}\left([0,T_1]; h^{s}\right)$. For $\tau \in [0,T_1]$, we can deduce from equation \eqref{truncnls} that
	\begin{align*}
		\left\|\tfrac{d}{d\tau} \{u_k(\tau)\}\right\|_{h^{s}} \leq C\|\{u_k(\tau)\}\|_{h^{s}}^3.
	\end{align*}
Let $\Lambda_{k,k_1,k_2,k_3}:=i(\lambda_{k_1}+\lambda_{k_2}-\lambda_{k_3}-\lambda_k)$.  Then
	\begin{align*}
				\frac{d^nu_k}{(d\tau)^n}(\tau)=  &\sum_{(k_1, k_2, k_3) \in \mathcal{R}^N(k)} \left(\tfrac{d}{d\tau}\right)^{n-1} \left[u_{k_1}u_{k_2}\bar{u}_{k_3}\right]\\
+  &\sum^*_{ |k_1|>N \ or \ |k_3|>N} \,\,\sum_{n-1=\alpha_1+\alpha_2}\Lambda_{k,k_1,k_2,k_3}^{\alpha_1}e^{i(\lambda_{k_1}+\lambda_{k_2}-\lambda_{k_3}-\lambda_k) \tau }\left(\tfrac{d}{d\tau}\right)^{\alpha_2}\left[ u_{k_1} u_{k_2}\bar{u}_{k_3}\right].
			\end{align*}
	
Our normal form reduction ensures that $|\Lambda_{k,k_1,k_2,k_3}|<1$ and since $|e^{i(\lambda_{k_1}+\lambda_{k_2}-\lambda_{k_3}-\lambda_k) \tau }|=1$, we have, for $n\geq 2$,
	\begin{align*}
		\left\|\tfrac{d^n}{(d\tau)^n} \{u_k(\tau)\}\right\|_{h^{s}} \leq \sum_{m=1}^n C^m\left( \prod_{j=1}^{m-1} (2j+1)\right)\|\{u_k(\tau)\}\|_{h^{s}}^{2m+1}.
	\end{align*}
Finally, 
	\begin{align*}
		\left\|\tfrac{d^n}{(d\tau)^n} \{u_k(0)\}\right\|_{h^{s}} \leq\sum_{m=1}^n \|\{u_k(0)\}\|_{h^{s}}m!(3C^{1/2}\|\{u_k(0)\}\|^2_{h^{s}})^{m}.
	\end{align*}
Therefore, $\{u_k\}$ is analytic for $\tau \leq T_2:= \frac{1}{6}C^{1/2}\|\{u_k(0)\}\|^{-2}_{h^{s}}$ using rearrangement.  Estimate \eqref{expbound} now follows from standard Taylor series estimates and setting $T= \tfrac{1}{2}\min(T_1, T_2)$.
\end{proof}

	\subsection{Taylor Series}\label{subsec.taylor}

Define and denote the $M$-box in $\mathbb{Z}^2$ by
	\begin{align*}
		Q_M = \left\{ k \in \mathbb{Z}^2~:~ \sup \{ |k^{(1)}|, |k^{(2)}|\} \leq M \right\}.
	\end{align*}
Using analyticity and a Taylor series expansion of each $u_k$ in $t$ we will show that  if the support of the 
$M$-box is centered at the origin of $\mathbb{Z}^2$, then for $k \in \mathbb{Z}^2 \setminus Q_M$ , $| u_k(t)|=0$ for all $t\in [0,T]$, where $T$ comes from Lemma \ref{lemmaT}. The primary barrier to the growth of the modes outside of $Q_M$ is that the resonant structure achieved after the normal form reduction prevents energy from passing from $Q_M$ to $Q_N \setminus Q_M$, where $N>M$. Moreover, the conservation of momentum prevents energy from jumping from the region $Q_M$ to $\mathbb{Z}^2 \setminus Q_N$.  Figure 4 demonstrates that for data concentrated in the $M$-box $Q_M$, the flow of the truncated system \eqref{truncnls}, does not allow the support of the solution to expand beyond $Q_M$.  The stability of the $Q_N \setminus Q_M$ and $\mathbb{Z}^2 \setminus Q_N$ regions are due to slightly different mechanisms as we will see in the proof of Lemma \ref{nocasc} below. These regions are both outlined in Figure 4 for this reason.

$$
\begin{array}{c} \mbox{Figure 4}\\
\includegraphics[scale=0.33]{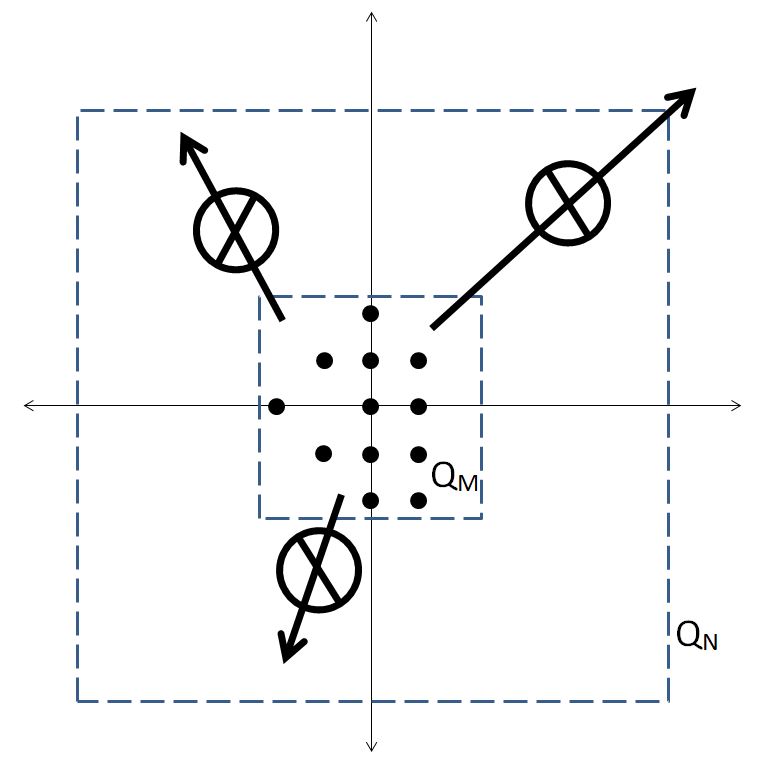}
\end{array}
$$

\begin{lemma}\label{nocasc}
	Let $M\in \mathbb{Z}_+$ and assume $\mbox{supp } \{u_k(0)\} \subset Q_M$ and let $\{u_k\}: [0,T] \rightarrow h^{s}$ be a solution to \eqref{truncnls}.  If $k \in  \mathbb{Z}^2 \setminus Q_M$ then for any $n \in \mathbb{Z}_+$,
	\begin{align*}
		  \frac{d^nu_k}{(dt)^n}(0)=0.
	\end{align*}
\end{lemma}

	\begin{proof}
		We will prove Lemma \ref{nocasc} by strong induction on the derivative $n$.  Let $k\in Q_N \setminus Q_M$ and $K:=\sup\{ |k^{(1)}|, |k^{(2)}|\}$.  Note that $K>M$.  By assumption, $u_k(0) = 0$ and 
			\begin{align*}
				i \dot{u}_k(0) =  \sum_{(k_1, k_2, k_3) \in \mathcal{R}^N(k)} u_{k_1}(0)u_{k_2}(0)\bar{u}_{k_3}(0)+  \sum^*_{|k_1|>N \ or \ |k_3|>N} u_{k_1}(0)u_{k_2}(0)\bar{u}_{k_3}(0).
			\end{align*}
		In the second sum, each component contains a factor, $u_j$, where $|j|>N$ and thus the second summation is equal to zero. The first summation can be handled by Lemma \ref{char} and  identity \eqref{resok}. Identity \eqref{resok} implies that every component of the first  sum contains a factor, $u_j$, such that $j \in Q_K \setminus Q_{K-1}$, which satisfies  $(Q_K \setminus Q_{K-1}) \cap Q_M= \emptyset$.  Thus, $\dot{u}_k(0)=0$. 
		
		\medskip
		
		For $k \in \mathbb{Z}^2 \setminus Q_N$, 
			\begin{align*}
				i \dot{u}_k(0) = \sum_{k_1+k_2=k_3+k} u_{k_1}(0)u_{k_2}(0)\bar{u}_{k_3}(0).
			\end{align*}
		Suppose, by contradiction that there exists $(k_1, k_2, k_3)$ such that $u_{k_1}(0)u_{k_2}(0)\bar{u}_{k_3}(0) \neq 0$ which implies $u_{k_i}(0)\neq 0$ for $i=1,2,3$.  Therefore, $k_i \in Q_M$ for $i =1,2,3$.  Since $k=k_1+k_2-k_3$, for $j=1,2$,   $k^{(j)}=k^{(j)}_1+k^{(j)}_2-k^{(j)}_3$ and thus
			\begin{align*}
				|k^{(j)}|=|k^{(j)}_1+k^{(j)}_2-k^{(j)}_3| \leq 3M.
			\end{align*}
However, $k \in \mathbb{Z}^2 \setminus Q_N$ implies that $|k^{(j)}|>N>3M$ for either $j=1$ or $j=2$. 
This is a contradiction, so $\dot{u}_k(0)=0$.

		Now assume $(\frac{d}{dt})^{m}u_k(0) = 0$ for all $k \in  \mathbb{Z}^2  \setminus Q_M$ and $m \in \{0, ..., n-1\}$.  Then for $k \in Q_N \setminus Q_M$,
			\begin{align*}
				\frac{d^nu_k}{(dt)^n}(0)= \left(\frac{d}{dt}\right)^{n-1} \left(  \sum_{(k_1, k_2, k_3) \in \mathcal{R}^N(k)} u_{k_1}u_{k_2}\bar{u}_{k_3}+  \sum^*_{ |k_1|>N \ or \ |k_3|>N} u_{k_1}u_{k_2}\bar{u}_{k_3} e^{i(\lambda_{k_1}+\lambda_{k_2}-\lambda_{k_3}-\lambda_k)\cdot} \right)(0)
			\end{align*}
Let $\Lambda_{k,k_1,k_2,k_3}:=i(\lambda_{k_1}+\lambda_{k_2}-\lambda_{k_3}-\lambda_k)$.  By distributing the derivatives we have
			\begin{align*}
				\frac{d^nu_k}{(dt)^n}(0)=  \sum_{(k_1, k_2, k_3) \in \mathcal{R}^N(k)} \,\,\sum_{n-1=\alpha_1+\alpha_2+\alpha_3}\left[\left(\tfrac{d}{dt}\right)^{\alpha_1} u_{k_1}\right](0)\left[\left(\tfrac{d}{dt}\right)^{\alpha_2} u_{k_2}\right](0)\left[\left(\tfrac{d}{dt}\right)^{\alpha_3} \bar{u}_{k_3}\right](0)\\
+  \sum^*_{ |k_1|>N \ or \ |k_3|>N} \,\,\sum_{n-1=\alpha_1+\alpha_2+\alpha_3+\alpha_4}\Lambda_{k,k_1,k_2,k_3}^{\alpha_1}\left[\left(\tfrac{d}{dt}\right)^{\alpha_2} u_{k_1}\right](0)\left[\left(\tfrac{d}{dt}\right)^{\alpha_3} u_{k_2}\right](0)\left[\left(\tfrac{d}{dt}\right)^{\alpha_4} \bar{u}_{k_3}\right](0).
			\end{align*}
The induction hypothesis and the constraints $(k_1, k_2, k_3) \in \mathcal{R}^N(k)$ or $|k_1|>N \ or \ |k_3|>N$ implies that for each summand in the above summation at least one factor is zero.   For $k \in \mathbb{Z}^2 \setminus Q_N$
			\begin{align*}
				\frac{d^nu_k}{(dt)^n}(0)&= \left(\frac{d}{dt}\right)^{n-1} \left(    \sum_{k_1+k_2=k_3+k} u_{k_1}u_{k_2}\bar{u}_{k_3} e^{i(\lambda_{k_1}+\lambda_{k_2}-\lambda_{k_3}-\lambda_k)\cdot} \right)(0)\\
				&=   \sum_{k_1+k_2=k_3+k}  \,\,\sum_{n-1=\alpha_1+\alpha_2+\alpha_3+\alpha_4}\Lambda_{k,k_1,k_2,k_3}^{\alpha_1}\left[\left(\tfrac{d}{dt}\right)^{\alpha_2} u_{k_1}\right](0)\left[\left(\tfrac{d}{dt}\right)^{\alpha_3} u_{k_2}\right](0)\left[\left(\tfrac{d}{dt}\right)^{\alpha_4} \bar{u}_{k_3}\right](0) .
			\end{align*}
The induction hypothesis and expansion concludes the proof. 
	\end{proof}
	
\begin{remark}In the square torus case, the initial step of the proof is invalid because for certain $k$ the vector field will have at least one nonzero component.  Any nonzero component in the initial vector field corresponds to a nonparallel resonance to which $k$ belongs. The existence of nonparallel resonances, such as $(k, k_1, k_2, k_3) \in \mathcal{R}$, leads to corresponding nonzero vector field components of the form $v_{k_1}v_{k_2}\bar{v}_{k_3}$, where $(k_1, k_2, k_3) \in \mathcal{R}^N(k)$, which in turn drive the dynamics in the work of Carles and Faou \cite{CF12}.
\end{remark}

Lemma \ref{nocasc} leads directly to the following corollary by Taylor expansion.
\begin{corollary}\label{nogrowth}
		Let $M \in \mathbb{Z}_+$ and assume $\mbox{supp } \{u_k(0)\} \subset Q_M$ and let $\{u_k\}: [0,T] \rightarrow h^{s}$ be an analytic solution to \eqref{truncnls}.  If $k \in  \mathbb{Z}^2 \setminus Q_M$, then 
	\begin{align*}
		  u_k(t)=0
	\end{align*}
		for $t \in [0, T]$.
	\end{corollary}

	\section{Proof of Theorem \ref{mainthm}}\label{sec.proof}

 Lemma \ref{lemmaT} limits the time of analyticity to being quadratic with respect to the size of the initial data. Therefore, when constructing the stability estimate we will only consider time in the interval $[0,\varepsilon^{-2}]$ where $\varepsilon>0$ is the size of the initial data. All initial data will also be assumed to have bounded Fourier support contained in the $M$-box, $Q_M$.

 Let $\varepsilon>0$ and $u(0)=z(0)=v(0)$, $\{v_k(0)\}_{k \in \mathbb{Z}^2}=\{u_k(0)\}_{k \in \mathbb{Z}^2}= \{z_k(0)\}_{k \in \mathbb{Z}^2}$ with $\|z(0)\|_s=\|u(0)\|_s=\|v(0)\|_s \leq \varepsilon$. Furthermore, let $\{z_k(t)\}$ be the solution to \eqref{newNLS}, the normalized Schr\"odinger equation, $\{v_k(t)\}$  be the solution to \eqref{gaugeNLS}, the {\it condensed} normalized Schr\"odinger system, and $\{u_k(t)\}$ be the solution to \eqref{truncnls}, the gauge transformation of \eqref{gaugeNLS}. 

 Corollary \ref{nogrowth}, implies that
	\begin{align*}
		|u_k(t)|=|v_k(t)|&=0  \mbox{ for } k \in \mathbb{Z}^2\setminus Q_M 
	\end{align*}
for $t <\varepsilon^{-2}$.

The  vector fields corresponding to $\{z_k(t)\}$, $\{v_k(t)\}$, and $\{u_k(t)\}$  are analytic vector-valued polynomials with nonlinearities  with zeroes of order three. We use Duhamel's formula which provides the equation
 	\begin{align*}
		&z_k(t) - v_k(t)\\ &= e^{i\lambda_k t}z_k(0)- e^{i\lambda_k t}v_k(0)\\ &\hspace{1cm}- e^{i\lambda_k t}\int_0^te^{i\lambda_k s}\left(X_{\mathscr{L}}(z(s))_k+X_{\mathscr{U}}(z(s))_k+X_{\mathscr{R}}(z(s))_k- X_{\mathscr{L}}(v(s))_k-X_{\mathscr{U}}(v(s))_k \right)\,ds\\
			&=e^{i\lambda_k t}\int_0^te^{i\lambda_k s}\big(X_{\mathscr{L}}(v(s)+(z(s)-v(s)))_k+X_{\mathscr{U}}(v(s)+(z(s)-v(s)))_k\\&\hspace{1cm}+X_{\mathscr{R}}(v(s)+(z(s)-v(s)))_k- X_{\mathscr{L}}(v(s))_k-X_{\mathscr{U}}(v(s))_k \big)\,ds
	\end{align*}
since $h^{s}$ for $s>1$ is an algebra and $X_{\mathscr{L}}$, $X_{\mathscr{U}}$ and $X_{\mathscr{R}}$ are  multilinear convolution polynomials, $X_{\mathscr{L}}(v(s))_k$ and $X_{\mathscr{U}}(v(s))_k $ cancel and thus
	\begin{align*}
		\|z(t)-v(t)\|_s &\lesssim \int_0^t  P(\|z(s)-v(s)\|_s, \|v(s)\|_s)+ \|v(s)\|_s^5\,ds\\
			&\lesssim t(C_s\varepsilon)^5+\int_0^t  P(\|z(s)-v(s)\|_s, \|v(s)\|_s)\,ds
	\end{align*}
where $P$ is a polynomial with a zero of order 3 at 0 and a zero of order one in $\|z(s)-v(s)\|_s$. Therefore, up to time $t \leq \varepsilon^{-2}$, $\|z(t)-v(t)\|_s \leq C_s^{5}\varepsilon^3$.  This implies that  
	\begin{align*}
		|z_k(t)|=|z_k(t)-v_k(t)|&\leq \langle k \rangle^{-s}C_s\varepsilon^3   
	\end{align*}
 for  $k \in \mathbb{Z}^2\setminus Q_M$ when $t \leq \varepsilon^{-2}$.


Let $\hat{\psi}_0=z(0)$ and $\hat{\psi}(t)$ be a solution to the equation generated by our original Hamiltonian \eqref{ham}.  Since $\|z(t)\|_s$ stays small with respect to $\varepsilon$,  we can choose $\varepsilon$ small enough so that $z(t)$ is in the image of the normal form transformation $\mathscr{T}$ from the Birkhoff Normal Form lemma, Lemma \ref{iterlemma}, for all $t<\varepsilon^{-2}$.  We then use the inverse of $\mathscr{T}$  to associate $z(t)$ to the trajectory, $\hat{\psi}(t)$, by $z=\mathscr{T}(\hat{\psi}(t))$ which further implies that $\|z(t)- \hat{\psi}(t)\|_s< C_{s, 3M, \omega} \varepsilon^3$.  
Therefore, for $k \in  \mathbb{Z}^2 \setminus Q_M$,
		\begin{align*}
			|\hat{\psi}_k(t)- z_k(t)|< \langle k \rangle^{-s}C_{s, 3M, \omega}\varepsilon^3 .
		\end{align*}  

For any $\gamma <3$, there exists a $\varepsilon_{s, \omega, M, \gamma}>0$ such that if $\varepsilon< \varepsilon_{s, \omega, M, \gamma}$, then 
		\begin{align*}
			 C_{s, 3M, \omega}\varepsilon^3 &< \frac{1}{2} \varepsilon^{\gamma} \mbox{ and } \\
			C_s\varepsilon^3 &< \frac{1}{2} \varepsilon^{\gamma}
		\end{align*}
which implies that $|z_k(t)|< \frac{1}{2}\varepsilon^{\gamma}$ and thus
		\begin{align*}
			|\hat{\psi}_k(t)| <\varepsilon^{\gamma}
		\end{align*}
for $t<\varepsilon^{-2}$.

This provides a contrasting result to that of Carles and Faou showing that there exists $\psi$ satisfying $|\hat{\psi}_k(t)|>\varepsilon^{2\eta+1}$ for $\eta \in (0,1)$, $t\approx \varepsilon^{-2}$, and $|k|\sim (\log \varepsilon^{-1})^{1/4}$.

\appendix

\section{Comparison with  the construction in \cite{CKSTT10} and \cite{GK12} }\label{app.A}

In this appendix we want to compare the construction of Colliander, Keel, Staffilani, Takaoka and Tao \cite{CKSTT10}, and the refinement of it by Guardia and Kaloshin  \cite{GK12}.

Consider again
\begin{align}\label{NLSApp}
		\left\{ \begin{array}{ll} i\partial_t \psi= \Delta \psi-|\psi|^2\psi, & x\in \mathbb{T}^2, t \in \mathbb{R}\\
						\psi(0)= \psi_0 \in H^s(\mathbb{T}^2),   \end{array} \right.
	\end{align}
where $\mathbb{T}^2$ is a square torus and $s>1$.
In   \cite{GK12} and \cite{GK17} (see also \cite{CKSTT10}), the authors  prove the following result.
	\begin{theorem}\label{GKThm}
		Let $s>1$. Then there exists $c > 0$ with the following property: for any small $\mu\ll1$ and any large
$K \gg 1$ there exists a  global solution $\psi(t, x)$ of \eqref{NLSApp} and a time $T$ satisfying
		\begin{align*}
			0 < T \leq e^{\left(K/\mu\right)^c}
		\end{align*}
such that
		\begin{align*}
			\|\psi(T)\|_{H^s(\mathbb{T}^2)} \geq K \hspace{.2cm} \mbox{ and } \hspace{.2cm} \|\psi(0)\|_{H^s(\mathbb{T}^2)} \leq \mu.
		\end{align*}

	\end{theorem}
As stated in Remark 1.2 of \cite{GK12}, in particular, one can show that  
	\begin{align*}
		\|\psi(T)\|_{H^s} \geq |n_1|^{2s}|\hat{\psi}_{n_1}(T)|^2+ |n_2|^{2s}|\hat{\psi}_{n_2}(T)|^2 \geq K.
	\end{align*}
It is important to mention that the initial data in Theorem \ref{GKThm} has bounded Fourier support which is precisely the type of initial data we consider throughout this paper.  The construction of the solution $\psi$ in the proof of the theorem above follows the construction  in \cite{CKSTT10}.

A key ingredient in \cite{CKSTT10}  that allows for the construction of solutions with growing Sobolev norms is the  identification of  a suitable subset of resonant frequencies, $\Lambda \subset\mathbb{Z}^2$, such that the 4-tuples, $(n_1, n_2, n_3, n_4) \in \Lambda^4$, form a subset of the set of non-degenerate resonant 4-tuples defined as 
	\begin{align*}
		\mathcal{A}:= \left\{ (n_1, n_2, n_3, n_4) \in (\mathbb{Z}^2)^4~:~ \begin{array}{c} n_1+n_2 = n_3+n_4 \\ |n_1|^2+|n_2|^2=|n_3|^2+|n_4|^2 \end{array}, n_1 \neq n_3, n_1 \neq n_4  \right\}.
	\end{align*}
Let us also define the set
	\begin{align*}
		\mathcal{A}(n):= \left\{ (n_1, n_2, n_3) \in (\mathbb{Z}^2)^3~:~ (n_1, n_2, n_3, n) \in \mathcal{A}.  \right\}.
	\end{align*}

Fix $N \gg 1$.  The set  $\Lambda$ is defined as a disjoint union of $N$ {\bf generations}:
	\begin{align*}
		\Lambda = \Lambda_1 \cup \cdots \cup \Lambda_N.	
	\end{align*}
Define a {\bf nuclear family} to be a rectangle $(n_1, n_2, n_3, n_4) \in \mathcal{A}$ such that $n_1$ and $n_2$ (known as {\bf parents}) belong to a generation $\Lambda_j$ and $n_3$ and $n_4$ (known as the {\bf children}) live in the next generation $\Lambda_{j+1}$.   The following conditions on $\Lambda$ are  imposed:

\begin{enumerate}
	\item[i] {\bf Closure} \hspace{.2cm} If $n_1, n_2, n_3 \in \Lambda$ and $(n_1, n_2, n_3) \in \mathcal{A}(n)$, then $n \in \Lambda$.
	\item[ii] {\bf Existence and uniqueness of spouse and children} \hspace{.2cm} For any $1\leq j\leq N$ and any $n_1 \in \Lambda_j$, there exists a unique nuclear family $(n_1, n_2, n_3, n_4)$ (up to trivial permutations) such that $n_1$ is a parent of this family. 
	\item[iii] {\bf Existence and uniqueness of sibling and parents} \hspace{.2cm} For any $1\leq j\leq N$ and any $n_3 \in \Lambda_{j+1}$, there exists a unique nuclear family $(n_1, n_2, n_3, n_4)$ (up to trivial permutations) such that $n_3$ is a child of this family.
	\item[iv] {\bf Nondegeneracy} \hspace{.2cm} The sibling of a frequency $n$ is never equal to its spouse.
	\item[v] {\bf Faithfulness} \hspace{.2cm} Apart from the nuclear families, $\Lambda$ does not contain any other rectangles.
\end{enumerate}
Unique to \cite{GK12}, is the {\bf No spreading condition}:
\begin{enumerate}
	\item[vi] {\bf No spreading condition} \hspace{.2cm} Consider $n \not\in \Lambda$.  Then, $n$ is vertex of at most two rectangles having two vertices in $\Lambda$ and two vertices out of $\Lambda$.
\end{enumerate}

The following Proposition holds:
	\begin{proposition}[Proposition 3.1 in \cite{GK12}]\label{Prop3.1}
		Let $K \gg 1$.  Then, there exists $N \gg 1$ and a set $\Lambda \subset \mathbb{Z}^2$, with 
			\begin{align*}
				\Lambda = \Lambda_1 \cup \cdots \cup \Lambda_N,	
			\end{align*}
		which satisfies conditions i - vi and also 
			\begin{align*}
				\frac{\sum_{n \in \Lambda_{N-1}} |n|^{2s}}{\sum_{n\in \Lambda_3}|n|^{2s}} \geq \frac{1}{2} 2^{(s-1)(N-4)} \geq K^2.	
			\end{align*}
Moreover, given any $R>0$ (which may depend on $K$), one  can ensure that each generation $\Lambda_j$ has $2^{N-1}$ disjoint frequencies $n$ satisfying $|n|\geq R$. 
	\end{proposition}

We will show that irrationality of the torus eliminates so many resonances that it becomes impossible to create families of resonances with growing generations. In fact, in the irrational case, $\sum_{n \in \Lambda_{j}} |n|^{2s}$  remains almost constant for $j \in \{ 1, ..., N\}$.  Recall that the resonances for the irrational torus can be written as
	\begin{align*}
		\mathcal{R}&:= \left\{ (k_1, k_2, k_3, k_4) \in (\mathbb{Z}^2)^4~:~ k_1+k_2=k_3+k_4\right\}\\
			&\hspace{.5cm} \bigcap \left\{ (k_1, k_2, k_3, k_4) \in (\mathbb{Z}^2)^4~:~ \lambda_{k_1}+\lambda_{k_2}=\lambda_{k_3}+\lambda_{k_4}\right\}\\
			&=\bigcap_{i=1}^2 \left\{ (k_1, k_2, k_3, k_4) \in (\mathbb{Z}^2)^4~:~ \begin{array}{c} k^{(i)}_1+k^{(i)}_2=k^{(i)}_3+k^{(i)}_4\\ \left(k^{(i)}_1\right)^2 +\left(k^{(i)}_2\right)^2 =\left(k^{(i)}_3\right)^2 + \left(k^{(i)}_4\right)^2 \end{array}\right\}.
	\end{align*}

In other words, if one were to construct a family, $\Lambda$, defined satisfying conditions i - vi, where the closure condition (condition i) is defined with respect to $\mathcal{R}$ instead of $\mathcal{A}$, then the conclusion of Proposition \ref{Prop3.1} does not follow.  Rather, the following proposition holds 
	\begin{proposition}
		Let $N \in \mathbb{Z}_+$.  Consider a set $\Lambda \subset \mathbb{Z}^2$, with 
			\begin{align*}
				\Lambda = \Lambda_1 \cup \cdots \cup \Lambda_N,	
			\end{align*}
		which satisfies conditions i - vi with $\mathcal{R}$ replacing $\mathcal{A}$.  Then
			\begin{align*}
				\frac{\sum_{n \in \Lambda_{k}} |n|^{2s}}{\sum_{n\in \Lambda_j}|n|^{2s}} \leq 2^s	
			\end{align*}
		for any $j, k \in \{ 1, ..., N\}$.
	\end{proposition}

	\begin{proof}
		Let $\Lambda= \Lambda_1 \cup \cdots \cup \Lambda_N$ be defined as above.  Due to the Faithfulness condition (condition v),  each 4-tuple $(n_1, n_2, n_3, n_4) \in \Lambda^4$ is a resonant rectangle, and due to the structure of $\mathcal{R}$, this rectangle must be axis parallel, since they also must be non-degenerate.   Let 
			\begin{align*}
				\Lambda_1 = \{ n_i = (a_i, b_i)\}_i
			\end{align*}
Existence and uniqueness of spouse and children (condition ii) implies that $\Lambda_1$ decomposes into a set of pairs of parents. 

Consider a pair of parents in $\Lambda_1$, $(n_{i},n_{j})$.  We can observe that the segment connecting $n_i$ to $n_j$ cannot be parallel to any axis.  If so, then without loss of generality assume that the segment connecting $n_i$ to $n_j$ is parallel to the $x$-axis.  Then there exists $b \in \mathbb{Z}$ such that $n_i=(a_i, b)= (a_j, b)=n_j$.  Since resonant rectangles must be axis parallel, the segment connecting any pair of children, $k_1, k_2 \in \Lambda_2$, must also be parallel to the $x$-axis.  The condition $n_i +n_j = k_1+k_2$, implies that $k_1=(k^{(1)}_1, b)$ and $k_2=(k^{(1)}_2, b)$ and thus $n_i, n_j, k_1, k_2$ are collinear which contradicts their forming of a non-degenerate rectangle. Therefore, we can assume that
 the segment connecting $n_i$ to $n_j$ cannot be parallel to any axis.

If $n_i, n_j \in \Lambda_1$ are a set of parents and the segment connecting $n_i$ to $n_j$ is not parallel to any axis, then there is only one possible choice among all points in $\mathbb{Z}^2$ for the pair of children that form the vertices of an axis-parallel rectangle in $\mathbb{Z}^2$. If $k_1, k_2 \in \Lambda_2$ are the children of $n_i=(a_i, b_i), n_j=(a_j, b_j)$, then 
		\begin{align*}
			k_1=(a_i, b_j)  \mbox{ and } k_2=(a_j, b_i).
		\end{align*}

Therefore, if $\Lambda_1 = \{ n_i = (a_i, b_i)\}_i$, then 
		\begin{align*}
			\Lambda_2= \{ (a_{i}, b_{\sigma(i)})\}_i
		\end{align*}
where $\sigma$ is a  permutation. By induction, for any $\ell \in \{2, ...., N\}$ there exists a permutation, $\tau_{\ell}$ such that 
		\begin{align*}
			\Lambda_{\ell}= \{ (a_{i}, b_{\tau_{\ell}(i)})\}_i.
		\end{align*}
Thus, 
		\begin{align*}
			\sum_{n \in \Lambda_{\ell}} |n|^{2s}= \sum_{i} (a^2_i+ b^2_{\tau_{\ell}(i)})^s &\leq 2^s \sum_{i} a^{2s}_i+ b^{2s}_{\tau_{\ell}(i)}\\
				&= 2^s \sum_{i} a^{2s}_i+ b^{2s}_{\tau_{j}(i)} \leq 2^s \sum_{i} (a^{2}_i+ b^{2}_{\tau_{j}(i)})^s=2^s \sum_{n \in \Lambda_{j}} |n|^{2s}
		\end{align*}
which finishes the proof.
	\end{proof}
Note that we did not use condition vi, so it is not necessary to assume $\Lambda$ satisfies condition vi. It is included to mirror Proposition \ref{Prop3.1}.

%
%


\end{document}